\documentclass{amsart}

\usepackage{amsmath,amssymb,url}

\numberwithin{equation}{section}

\theoremstyle{plain}

\newtheorem{theorem}{Theorem}
\newtheorem{corollary}{Corollary}
\newtheorem{lemma}{Lemma}
\newtheorem{claim}{Claim}
\theoremstyle{definition}
\newtheorem{remark}{Remark}

\newcommand{\w}{\omega}
\newcommand{\I}{\mathcal I}

\newcommand{\A}{\mathcal A}
\newcommand{\U}{\mathcal U}
\newcommand{\F}{\mathcal F}
\newcommand{\Tau}{\mathcal T}
\newcommand{\id}{\mathrm{id}}

\begin{document}


\title[Large free sets in powers of universal algebras]{Large free sets in powers of universal algebras}

\author[T. Banakh]{Taras Banakh}
\email{t.o.banakh@gmail.com}
\address{Ivan Franko University of Lviv (Ukraine) and Jan Kochanowski Uniwersity in Kielce (Poland)}

\author[A. Bartoszewicz]{Artur Bartoszewicz}
\email{arturbar@p.lodz.pl}
\urladdr{http://im.p.lodz.pl/struktura/pracownicy.html?pracownik=116}
\address{Institute of Mathematics, Technical University of \L\'od\'z, W\'olcza\'nska 215, 93-005
\L\'od\'z, Poland}

\author[Sz.~G\l\c ab]{Szymon G\l\c ab}
\email{szymon.glab@p.lodz.pl}
\urladdr{http://im0.p.lodz.pl/~sglab/}
\address{Institute of Mathematics, Technical University of \L\'od\'z, W\'olcza\'nska 215, 93-005
\L\'od\'z, Poland}

\thanks{The first author has been partially financed by NCN grant DEC-2012/07/D/ST1/02087. The second and the third authors have been supported by the Polish Ministry of Science and Higher Education Grant No N N201 414939 (2010-2013).}




\begin{abstract}
 We prove that for each universal algebra $(A,\A)$ of cardinality  $|A|\ge 2$ and an infinite set $X$ of cardinality $|X|\ge|\A|$, the $X$-th power $(A^X,\A^X)$ of the algebra $(A,\A)$ contains a free subset $\F\subset A^X$ of cardinality $|\F|=2^{|X|}$. This generalizes the classical Fichtenholtz-Kantorovitch-Hausdorff result on the existence of an independent family $\I\subset\mathcal P(X)$ of cardinality $|\I|=|\mathcal P(X)|$ in the Boolean algebra $\mathcal P(X)$ of subsets of an infinite set $X$.
\end{abstract}

\maketitle


The classical Fichtenholz-Kantorovitch-Hausdorff Theorem \cite{FK}, \cite{Ha} (see also \cite[9.21]{HBA} or \cite[p.83]{JW}) says that the power-set $\mathcal P(X)$ of any infinite set $X$ contains an independent family $\I\subset\mathcal P(X)$ of cardinality $|\I|=|\mathcal P(X)|$. The independence of $\I$ means that for any disjoint finite subsets $\A,\mathcal B\subset\I$ the intersection $(\bigcap_{A\in \A}A)\cap(\bigcap_{B\in\mathcal B}X\setminus B)$ is not empty.
In this paper we shall use this Fichtenholz-Kantorovitch-Hausdorff Theorem to prove the existence of free sets of large cardinality in powers of universal algebras.

First we recall the necessary definitions from the theory of universal algebras \cite{BS}.
By an {\em algebraic operation} on a set $A$ we understand a function $\alpha:A^{n_{\alpha}}\to A$ defined on a finite power of the set $A$. The number $n_{\alpha}\in\w$ is called the {\em arity} of the algebraic operation $\alpha$.
Let $(A,\A)$ be a {\em universal algebra}, i.e., a set $A$ endowed with a family of algebraic operations $\A$. A subset $S\subset A$ is called a {\em subalgebra} of the algebra $(A,\A)$ if $\alpha(S^{n_{\alpha}})\subset S$ for each algebraic operation $\alpha\in\A$. In this case $S$ is a universal algebra endowed with the family $\A|S=\{\alpha|S^{n_{\alpha}}\}_{\alpha\in\A}$ of restricted algebraic operations. A function $h:S\to A$ is called a {\em homomorphism} if it preserves algebraic operations in the sense that
$h(\alpha(x_1,\dots,x_{n_{\alpha}}))=\alpha(h(x_1),\dots,h(x_{n_{\alpha}}))$ for any algebraic operation $\alpha\in\A$ and points $x_1,\dots,x_{n_{\alpha}}\in S$.

Each subset $B\subset A$ is contained in the smallest subalgebra $\bar B\subset A$ called {\em the subalgebra generated by $B$}. We shall say that the subalgebra $\bar B$ is {\em freely generated by $B$} (or else that the subset $B$ is {\em free} in $A$) if every function $f:B\to A$ can be extended to a homomorphism $\bar f:\bar B\to A$. Free sets in universal algebras were introduced by E.~Marczewski in \cite{M1}, \cite{M2}, \cite{M3} and later studied in \cite{Nar1}, \cite{Nar2}, \cite{Glazek}, \cite{BF}, \cite[\S9-10]{HBA}, \cite{Gould}, \cite{AF}.

For a set $X$ by $(A^X,\A^X)$ we shall denote the $X$-th power of the universal algebra $(A,\A)$. This is the set $A^X$ of all functions from $X$ to $A$ endowed with the family $\A^X=\{\alpha^X\}_{\alpha\in\A}$ of algebraic operations $\alpha^X:(A^X)^{n_{\alpha}}\to A^X$, $\alpha\in\A$, defined by $$\alpha^X(f_1,\dots,f_{n_{\alpha}})(x)=\alpha(f_1(x),\dots,f_{n_\alpha}(x))\mbox{ \ for \ }(f_1,\dots,f_{n_{\alpha}})\in (A^X)^{n_\alpha}.$$

The main result of this paper is the following theorem. A special case of this theorem was proved in  \cite{BGP}.

\begin{theorem}\label{main} For any universal algebra $(A,\A)$ of cardinality $|A|\ge 2$ and any infinite set $X$ of cardinality $|X|\ge|\A|$, the universal algebra $(A^X,\A^X)$ contains a free subset $\F\subset A^X$ of cardinality $|\F|=2^{|X|}$.
\end{theorem}

The proof of this theorem uses the Fichtenholz-Kantorovitch-Hausdorff Theorem, which in turn, can be considered as a special case of Theorem~\ref{main}.

\begin{corollary}[Fichtenholtz-Kantorovitch-Hausdorff] For any infinite set $X$ the Boolean algebra $\mathcal P(X)$ contains an independent subset $\I\subset\mathcal P(X)$ of cardinality $|\I|=|\mathcal P(X)|$.
\end{corollary}

\begin{proof} Using characteristic functions, identify the Boolean algebra $\mathcal P(X)$ with the $X$-th power $\mathbf 2^X$ of a two-element Boolean algebra $\mathbf 2$. By Theorem~\ref{main}, the Boolean algebra $\mathcal P(X)$ contains a free subset $\mathcal I\subset\mathcal P(X)$ of cardinality $|\mathcal I|=|\mathcal P(X)|$. By Proposition 9.4 of \cite{HBA}, the family $\mathcal I$ is independent.
\end{proof}

\begin{remark} In light of the Fichtenholz-Kantorovitch-Hausdorff Theorem it is interesting to remark that in certain models of ZFC the smallest cardinality $\mathfrak i$ of a maximal independent subset in the Boolean algebra $\mathcal P(\w)$ is strictly smaller than the cardinality of continuum $2^\w$, see \cite[p.480]{Blass}, \cite{Va}.  This fact witnesses that Theorem~\ref{main} cannot be proved by a maximality argument using Zorn's Lemma.
\end{remark}

\section{Preliminaries}

In this section we present some auxiliary results, necessary for the proof of Theorem~\ref{main}.
In particular, in Lemma~\ref{l3} we prove a convenient criterion for recognizing free sets in universal algebras. First we introduce some definition.

A family $\A$ of operations on a set $A$ is called
\begin{itemize}
\item {\em unital} if $\A$ contains the identity operation $\id_A:A^1\to A$, $\id_A:(x_1)\mapsto x_1$;
\item  {\em stable under substitutions} (briefly, {\em substitution-stable}) if for any algebraic operation $\alpha\in\A$ and a function $s:\{1,\dots,n_{\alpha}\}\to \{1,\dots,m\}$ the algebraic operation $\alpha\circ A^s:A^m\to A$, $\alpha\circ A^s:(x_1,\dots,x_m)\mapsto \alpha(x_{s(1)},\dots,x_{s(n_{\alpha})})$, belongs to $\A$;
\item {\em stable under compositions} if for any algebraic operations $\beta\in\A$ and $\alpha_1,\dots,\alpha_{n_{\beta}}\in\A$ of the same arity $n=n_{\alpha_i}$, $i\le n_\beta$, the algebraic operation $\beta(\alpha_1,\dots,\alpha_{n_{\beta}}):A^n\to A$, $\beta(\alpha_1,\dots,\alpha_{n_{\beta}}):(x_1,\dots,x_n)\mapsto \beta(\alpha_1(x_1,\dots,x_n),\dots,\alpha_{n_{\beta}}(x_1,\dots,x_n))$ belongs to $\A$;
\item a {\em clone} if $\A$ is unital and stable under substitutions and compositions.
\end{itemize}

Let $(A,\A)$ be a universal algebra and $\bar \A$ be the smallest clone containing the operation family $\A$. It is clear that $|\bar\A|\le\max\{|\A|,\aleph_0\}$. It is standard to prove (see \cite[2.2]{BBG} for details) that for each subset $B\subset A$ the subalgebra $\bar B$ generated by $B$ coincides with the set $\bar\A(B)=\bigcup_{\alpha\in\bar\A}\alpha(B^{n_{\alpha}})$. This means that for each point $y\in\bar B$ we can find an algebraic operation $\alpha\in\bar\A$ and points $x_1,\dots,x_{n_{\alpha}}\in B$ such that $y=\alpha(x_1,\dots,x_{n_\alpha})$. We can additionally assume that the points $x_1,\dots,x_{n_{\alpha}}$ are pairwise distinct:

\begin{lemma}\label{l1} For each point $y\in\bar B$ there are an algebraic operation $\alpha\in\bar\A$ and pairwise distinct points $x_1,\dots,x_{n_{\alpha}}\in B$ such that $y=\alpha(x_1,\dots,x_{n_{\alpha}})$.
\end{lemma}

\begin{proof} Since $y\in\bar B=\bar \A(B)$, there are an algebraic operation $\beta\in\A$ and points $y_1,\dots,y_{n_{\beta}}\in B$ such that $y=\beta(y_1,\dots,y_{n_{\beta}})$. Let $\{y_1,\dots,y_{n_\beta}\}=\{x_1,\dots,x_n\}$ where the points $x_1,\dots,x_n$ are pairwise distinct.
Let $s:\{1,\dots,n_{\beta}\}\to\{1,\dots,n\}$ be the unique map such that $y_i=x_{s(i)}$ for all $1\le i\le n_{\beta}$, and $A^s:A^n\to A^{n_{\beta}}$, $A^s:(a_1,\dots,a_n)\mapsto (a_{s(1)},\dots,a_{s(n_{\beta})})$, be the corresponding substitution operator. Then for the algebraic operation $\alpha=\beta\circ A^s\in\bar\A$, we get $y=\beta(x_1,\dots,x_{n_{\beta}})=\beta\circ A^s(x_1,\dots,x_n)=\alpha(x_1,\dots,x_n)$.
\end{proof}

Now we shall elaborate some tools for recognizing free sets in universal algebras.
Our first characterization follows immediately from Lemma~\ref{l1} and was noticed by Marczewski in \cite{M2}.

\begin{lemma}\label{l2} A subset $B\subset A$ of a universal algebra $(A,\A)$ is free if and only if for any  algebraic operations $\alpha,\beta\in\bar\A$ and sequences $(x_1,\dots,x_{n_{\alpha}})\in B^{n_{\alpha}}$, $(y_1,\dots,y_{n_{\beta}})\in B^{n_{\beta}}$ the equality $\alpha(x_1,\dots,x_{n_{\alpha}})=\beta(y_1,\dots,y_{n_{\beta}})$ implies that $\alpha(f(x_1),\dots,f(x_{n_{\alpha}}))=\beta(f(y_1),\dots,f(y_{n_{\beta}}))$ for any function $f:B\to A$.
\end{lemma}

This characterization can be improved as follows.

\begin{lemma}\label{l3} A subset $B\subset A$ of a universal algebra $(A,\A)$ is free if and only if for any distinct algebraic operations $\alpha,\beta\in\bar\A$ of the same arity $n=n_{\alpha}=n_{\beta}$ the inequality $\alpha(x_1,\dots,x_n)\ne\beta(x_1,\dots,x_n)$ holds for any pairwise distinct points $x_1,\dots,x_n\in B$.
\end{lemma}

\begin{proof} To prove the ``only if'' part, assume that the set $B$ is free. Fix two distinct algebraic operations $\alpha,\beta\in\bar \A$ of the same arity $n=n_\alpha=n_{\beta}$. We need to show that $\alpha(x_1,\dots,x_n)\ne\beta(x_1,\dots,x_n)$ for any pairwise distinct points $x_1,\dots,x_n\in B$.

Since $\alpha\ne \beta$, there is a sequence $(y_1,\dots,y_n)\in A^n$ such that $\alpha(y_1,\dots,y_n)\ne\beta(y_1,\dots,y_n)$. Since the points $x_1,\dots,x_n$ are pairwise distinct, we can choose a function $f:B\to A$ such that $f(x_i)=y_i$ for all $i\le n$. Since the set $B$ is free, the function $f$ extends to a homomorphism $\bar f:\bar B\to A$. Then
 $$
 \begin{aligned}
 \bar f(\alpha(x_1,\dots,x_n))&=\alpha(\bar f(x_1),\dots,\bar f(x_n))=\alpha(y_1,\dots,y_n)\ne\beta(y_1,\dots,y_n)=\\
 &=\beta(\bar f(x_1),\dots,\bar f(x_n))=\bar f(\beta(x_1,\dots,x_n))
  \end{aligned}
 $$which implies that $\alpha(x_1,\dots,x_n)\ne \beta(x_1,\dots,x_n)$.
\smallskip

To prove the ``if'' part, assume that the set $B$ is not free. Applying Lemma~\ref{l2}, find algebraic operations $\alpha,\beta\in\bar\A$ and sequences $(x_1,\dots,x_{n_{\alpha}})\in B^{n_{\alpha}}$, $(y_1,\dots,y_{n_{\beta}})\in B^{n_{\beta}}$ such that $\alpha(x_1,\dots,x_{n_{\alpha}})=\beta(y_1,\dots,y_{n_{\beta}})$ and\break $\alpha(f(x_1),\dots,f(x_{n_{\alpha}}))\ne \beta(f(y_1),\dots,f(y_{n_{\beta}}))$ for some function $f:B\to A$. For the set $Z=\{x_1,\dots,x_{n_{\alpha}}\}\cup\{y_1,\dots,y_{n_{\beta}}\}$ choose pairwise distinct points $z_1,\dots,z_n\in Z\subset B$ such that $\{z_1,\dots,z_n\}=Z$.
Let $s:\{1,\dots,n_{\alpha}\}\to\{1,\dots,n\}$ and $t:\{1,\dots,n_{\beta}\}\to\{1,\dots,n\}$ be the unique functions such that $x_i=z_{s(i)}$ and $y_j=z_{t(j)}$ for all  $i\in\{1,\dots,n_{\alpha}\}$ and $j\in\{1,\dots,n_{\beta}\}$. The functions $s$ and $t$ induce the substitution operators $A^s:A^n\to A^{n_{\beta}}$, $A^s:(a_1,\dots,a_n)\mapsto (a_{s(1)},\dots,a_{s(n_{\alpha})})$, and $A^t:A^n\to A^{n_{\beta}}$, $A^t:(a_1,\dots,a_n)\mapsto (a_{t(1)},\dots,a_{t(n_{\beta})})$. Since the clone $\bar \A$ is substitution-stable, the algebraic operations $\alpha\circ A^s,\beta\circ A^t:A^n\to A$ belong to $\bar \A$ and have the same arity $n$. Moreover,
$$
\begin{aligned}
\alpha\circ A^s(z_1,\dots,z_n)&=\alpha(z_{s(1)},\dots,z_{s(n_{\alpha})})=\alpha(x_1,\dots,x_{n_{\alpha}})=\\
&=\beta(y_1,\dots,y_{n_{\beta}})=\beta(z_{t(1)},\dots,z_{t(n_{\beta})})=\beta\circ A^t(z_1,\dots,z_n).
\end{aligned}
$$
It remains to check that the algebraic operations $\alpha\circ A^s$ and $\beta\circ A^t$ are distinct.
For this observe that
$$
\begin{aligned}
\alpha\circ A^s&(f(z_1),\dots,f(z_n)) =\alpha(f(z_{s(1)}),\dots,f(z_{s(n_{\alpha})}))=\\
&=\alpha(f(x_1),\dots,f(x_{n_{\alpha}}))\ne\beta(f(y_1),\dots,f(y_{n_{\beta}}))=\\
&=\beta(f(z_{t(1)}),\dots,f(z_{t(n_{\beta})}))=\beta\circ A^t(f(z_1),\dots,f(z_n)).
\end{aligned}
$$
\end{proof}

\section{Proof of Theorem~\ref{main}}

Let $(A,\A)$ be a universal algebra of cardinality $|A|\ge 2$ and $X$ be an infinite set of cardinality $|X|\ge|\A|$. It follows that the clone $\bar\A$ generated by the family $\A$ has cardinality $|\bar\A|\le\max\{|\A|,\aleph_0\}\le|X|$.

For every $n\in\w$ consider the family of triples
$$\Tau_n=\{(\alpha,\beta,s)\in\bar\A\times\bar\A\times X^n:\alpha\ne\beta,\;\;n_{\alpha}=n=n_{\beta}\}$$and observe that it has cardinality $|\Tau_n|\le |\bar\A\times \bar\A\times X^n|\le |X|$.
Then the union $\Tau=\bigcup_{n\in\w}\Tau_n$ also has cardinality $|\Tau|\le|X|$ and
 hence admits an enumeration $\Tau=\{(\alpha_x,\beta_x,s_x):x\in X\}$ by points of the set $X$.
For every $x\in X$ let $n_x=n(\alpha_x)=n(\beta_x)$ be the arity of the algebraic operations $\alpha_x$ and $\beta_x$. Since $\alpha_x\ne\beta_x$, we can choose a point $p_x=(p_x(1),\dots,p_x(n_x))\in A^{n_x}$ such that $\alpha_x(p_x)\ne \beta_x(p_x)$.

The Fichtenholtz-Kantorovitch-Hausdorff Theorem~17.20 \cite[p.83]{JW}, yields an independent subfamily $\U\subset \mathcal P(X)$ of cardinality $|\U|=2^{|X|}$.

For each set $U\in\U$ and a point $x\in X$ consider the sequence $s_x=(s_x(1),\dots,s_x(n_x))\in X^{n_x}$ and the set $U_x=\{i\in\{1,\dots,n_x\}:s_x(i)\in U\}$.

For each set $U\in\U$ choose a function $f_U:X\to A$ such that $\{f_U(x)\}=p_x(U_x)$ for each point $x\in X$ with $|U_x|=1$.

We claim that the set $\F=\{f_U\}_{U\in\U}\subset A^X$ has cardinality $|\F|=2^{|X|}$ and is free in the algebra $(A^X,\A^X)$.

\begin{claim} The function $f:\U\to \F$, $f:U\mapsto f_U$, is bijective and hence $|\F|=|\U|=2^{|X|}$.
\end{claim}

\begin{proof} Given two distinct sets $U,V\in\U$, we should prove that ${f_U\ne f_V}$.
Since the family $\U$ is independent, there are points $s(1)\in U\setminus V$ and $s(2)\in V\setminus U$. The points $s(1),s(2)$ form a sequence $s=(s(1),s(2))\in X^2$.
Since the clone $\bar\A$ contains the identity operation $A^1\mapsto A$, $(x)\mapsto x$, and is substitution-stable, the operations $\alpha:A^2\to A$, $\alpha(x,y)\mapsto x$ and $\beta:A^2\to A$, $\beta:(x,y)\mapsto y$ belong to the clone $\bar\A$. The triple $(\alpha,\beta,s)$ belongs to the family $\Tau$ and hence $(\alpha,\beta,s)=(\alpha_x,\beta_x,s_x)$ for some $x\in X$. It follows that $p_x(1)=\alpha(p_x)\ne\beta(p_x)=p_x(2)$ and the sets $U_x=\{i\in\{1,2\}:s_x(i)\in U\}=\{1\}$ and $V_x=\{i\in\{1,2\}:s_x(2)\in V\}=\{2\}$ are singletons. Then $f_U(x)=p_x(1)\ne p_x(2)=f_V(x)$, which means that $f_U\ne f_V$.
\end{proof}

\begin{claim} The set $\F$ is free in the universal algebra $(A^X,\A^X)$.
\end{claim}

\begin{proof} By Lemma~\ref{l3}, it suffices to prove that for any distinct algebraic operations $\alpha,\beta\in\bar\A$ of the same arity $n=n_{\alpha}=n_{\beta}$ and pairwise distinct functions $f_{1},\dots,f_{n}\in \F$ we get $\alpha^X(f_{1},\dots,f_{n})\ne\beta^X(f_{1},\dots,f_{n})$.

For every $i\in\{1,\dots,n\}$ find a set $U_i\in\U$ such that $f_i=f_{U_i}$.
The independence of the family $\U$ guarantees the existence of a sequence $s=(s(1),\dots,s(n))\in X^n$ such that  $s(i)\in U_i\setminus U_j$ for all $j\ne i$.

The triple $(\alpha,\beta,s)$ belongs to the family $\Tau$ and hence is equal to $(\alpha_x,\beta_x,s_x)$ for some point $x\in X$.
The definition of the functions $f_{U_i}$ and the choice of the sequence $s$ guarantees that $f_i(x)=f_{U_i}(x)=p_x(i)$ for all $x\in X$ and $i\in\{1,\dots,n\}$. The choice of the sequence $p_x=(p_x(1),\dots,p_x(n))$ guarantees that  $\alpha(p_x(1),\dots,p_x(n))\ne\beta(p_x(1),\dots,p_x(n))$.
Then
$$
\begin{aligned}
\alpha^X&(f_{1},\dots,f_n)(x)=\alpha(f_{1}(x),\dots,f_{n}(x))=\alpha(p_x(1),\dots,p_x(n))\ne\\
&\ne \beta(p_x(1),\dots,p_x(n))=\beta(f_{1}(x),\dots,f_{n}(x))=\beta^X(f_1,\dots,f_n)(x)
\end{aligned}
$$which means that $\alpha^X(f_1,\dots,f_n)\ne\beta^X(f_1,\dots,f_n)$.
\end{proof}

\section{Acknowledgment}

The authors express their sincere thanks to Jo\~ao Ara\'ujo for pointing us relevant references in Independence Theory and to the anonymous referee for valuable suggestions on improvement of the presentation of the paper (whose initial version can be seen in \cite{BBG}).

\end{document}